\documentclass[12pt,a4paper]{article}
\usepackage[dvips]{graphicx,color} 
\usepackage[T1]{fontenc}  
\usepackage{times}    
\usepackage{enumerate,amsmath,amsthm,mathrsfs,citesort}
\usepackage{amsfonts}
\usepackage{amssymb}
\usepackage[a4paper,left=2.5cm,right=2cm,top=2.5cm,bottom=2.5cm]{geometry}
\usepackage{picins}
\usepackage{amstext}
\usepackage{array}
\usepackage[disable]{todonotes} 
\usepackage[symbol]{footmisc}
\newcounter{count}
\newtheoremstyle{bthm}{\baselineskip}{\baselineskip}{\slshape}{}{\bfseries}{}{ }{}
\newtheoremstyle{bex}{\baselineskip}{\baselineskip}{}{}{\sffamily}{:}{\newline }{}
\theoremstyle{bthm}
\newtheorem{theorem}{Theorem}[section]
\newtheorem{corollary}[theorem]{Corollary}
\newtheorem{question}[theorem]{Question}
\newtheorem{lemma}[theorem]{Lemma}
\newtheorem{proposition}[theorem]{Proposition}

\theoremstyle{bex}

\begin{document}
\begin{titlepage}
\title{An optimal chromatic bound for the class of $\{P_3\cup 2K_1,\overline{P_3\cup 2K_1}\}$-free graphs}
\author{Athmakoori Prashant$^{1,}$\footnote{The author's research was supported by the Council of Scientific and Industrial Research,  Government of India, File No: 09/559(0133)/2019-EMR-I.} and S. Francis Raj$^{2}$}
\date{{\footnotesize Department of Mathematics, Pondicherry University, Puducherry-605014, India.}\\
{\footnotesize$^{1}$: 11994prashant@gmail.com\, $^{2}$: francisraj\_s@pondiuni.ac.in\ }}
\maketitle
\renewcommand{\baselinestretch}{1.3}\normalsize
\begin{abstract}
In 1987, A. Gy\'arf\'as in his paper ``Problems from the world surrounding perfect graphs'' posed the problem of determining the smallest $\chi$-binding function for $\mathcal{G}(F,\overline{F})$, when $\mathcal{G}(F)$ is $\chi$-bounded.
So far the problem has been attempted for only forest $F$ with four or five vertices.
In this paper, we address the case when $F=P_3\cup 2K_1$ and show that if $G$ is a $\{P_3\cup 2K_1,\overline{P_3\cup 2K_1}\}$-free graph with $\omega(G)\neq 3$, then it admits $\omega(G)+1$ as a $\chi$-binding function.
Moreover, we also construct examples to show  that this bound is tight for all values of $\omega\neq 3$.
\end{abstract}
\section{Introduction and Motivation}\label{intro}
All graphs considered in this paper are finite and without loops or parallel edges.
For notations and terminology not mentioned here we refer to \cite{bondy2008graph}.
For any positive integer $k$, a proper coloring using $k$-colors is a mapping $f:V(G)\rightarrow \{1,2,\ldots,k\}$ such that $f(u)\neq f(v)$, whenever $uv\in E(G)$. A graph $G$ is said to be \emph{$k$-colorable} if there exists a proper coloring using $k$-colors. The smallest $k$ for which $G$ is $k$-colorable is known as the \emph{chromatic number} of $G$ and is denoted by $\chi(G)$.
For any graph $G$, we write  $G'\sqsubseteq G$ if $G'$ is an induced subgraph of $G$. We say that a graph $G$ is \emph{perfect} if $\chi(G')=\omega(G')$ for every $G'\sqsubseteq G$.\\
\\
\textbf{Gy\'arf\'as's  Question}\\
Given a family of graphs $\mathcal{F}$, a graph $G$ is said to be \emph{$\mathcal{F}$}-free if it contains no induced subgraph which is isomorphic to a graph in $\mathcal{F}$.
Let $\mathcal{G(F)}$ denote the family of $\mathcal{F}$-free graphs.
A family of graphs $\mathcal{G}$ is said to be \emph{$\chi$-bounded} with a $\chi$-binding function $f$, if $\chi(G')\leq f(\omega(G'))$ holds whenever $G'$ is an induced subgraph of $G\in \mathcal{G}$.
A. Gy\'arf\'as in his seminal paper \cite{gyarfas1987problems}, introduced the concept of $\chi$-bound and $\chi$-binding function to relax the concept of perfection.
In \cite{gyarfas1987problems}, he also conjectured that if $F$ is a forest then $\mathcal{G}(F)$ is $\chi$-bounded and proved the conjecture for $P_k$-free graph.
S. Gravier et al. in \cite{gravier2003coloring}, slightly improved the bound given by A. Gy\'arf\'as and showed that if $G$ is a $P_k$-free graph then $\chi(G)\leq (k-2)^{\omega(G)-1}$.
Since $(P_3\cup 2K_1)\sqsubseteq P_7$, it follows from the preceeding result that every $(P_3\cup 2K_1)$-free graph $G$ satisfies $\chi(G)\leq 5^{\omega(G)-1}$, an exponential bound.
A. Gy\'arf\'as has also proposed a number of problems and conjectures that are still open and has received wide attention since its introduction. See for instance, \cite{schiermeyer2019polynomial,scott2020survey,randerath2004vertex}. In this paper, we shall study one of these problems which is mentioned below.  
\begin{question}\cite{gyarfas1987problems}
For a fixed forest $F$, what is the optimal $\chi$-binding function for $\mathcal{G}(F,\overline{F})$, when $\mathcal{G}(F)$ is $\chi$-bounded.
\end{question}
Question 1.1 is still open and seems to be hard even when $F$ is a simple type of forest.
So far the optimal $\chi$-binding function for the class of $\{F,\overline{F}\}$-free graphs has been found for forest $F$ on four or five vertices. 
The details of the known bounds are listed in Table \ref{mytable}.
One can see that when $F$ is a forest on four vertices, $4K_1$ is the only graph for which optimal $\chi$-bound has to be determined and for five vertices, $K_1+4K_1, (K_1+3K_1)\cup K_1, P_3\cup 2K_1, 2K_2\cup K_1, P_2\cup 3K_1, 5K_1$ are the graphs for which the optimal $\chi$-bound is yet to be determined.

\begin{table}
\centering
\scriptsize
\setlength{\tabcolsep}{40pt}
\begin{tabular}{|c c c|}
\hline
F & $\chi$-bound for $\{F,\overline{F}\}$-free graphs & \\
\hline
$K_2\cup 2K_1$ & $\max\{3,\omega\}$ \cite{gyarfas1987problems} & optimal \\
$2K_2$ & $\omega+1$ \cite{gyarfas1987problems} & optimal\\
$P_3\cup K_1$ & $\max\{3,\omega\}$ \cite{gyarfas1987problems} & optimal \\
$K_1+3K_1$ & $\lfloor\frac{3\omega(G)}{2}\rfloor$ \cite{gyarfas1987problems} & optimal \\
$P_4$ & $\omega$ \cite{gyarfas1987problems} & optimal \\
$P_5$& $\binom{\omega+1}{2}$ \cite{fouquet1995graphs} & not optimal \\
$P_4\cup K_1$& $\lceil\frac{5\omega(G)}{4}\rceil$ \cite{karthick2018coloring} & optimal \\
$fork$& $2\omega$ \cite{chudnovsky2020excluding} & asymptotically tight\\
$P_3\cup P_2$& $\max\{\omega(G)+3,\lfloor\frac{3\omega(G)}{2}\rfloor-1\}$ \cite {char2022optimal} & optimal\\
\hline
\end{tabular}
\caption{The table denotes the $\chi$-binding functions for the class of $(F,\overline{F})$-free graphs.}
\label{mytable}
\end{table}

Our goal in this paper is to find the smallest $\chi$-binding function for the class of $(F,\overline{F})$-free graphs, when $F=P_3\cup 2K_1$. 
In particular, we shall show that if $G$ is a $\{P_3\cup 2K_1,\overline{P_3\cup 2K_1}\}$-free graph with $\omega(G)\neq 3$, then $\chi(G)\leq \omega+1$. When $\omega=3$ we have shown that $\chi(G)\leq 7$. 
Also, we have constructed examples to establish that this bound is tight for every $\omega\neq 3$ and give an example of a $\{P_3\cup 2K_1,\overline{P_3\cup 2K_1}\}$-free graph with $\omega(G)=3$ and $\chi(G)=6$.   
The results in this paper use a partition of the vertex set which was initially defined by S. Wagon in \cite{wagon1980bound} and later modified by A. P. Bharathi and S. A. Choudum  in \cite{bharathi2018colouring}. This  has been given in Section \ref{2}.

Some of the graphs that are considered as forbidden induced subgraphs in this paper are given in Figure \ref{somesplgraphs}.
\begin{figure}[t]
	\centering
		\includegraphics[width=0.6\textwidth]{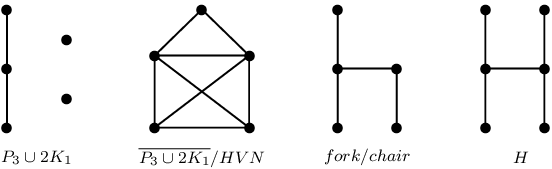}
	\caption{Some special graphs}
	\label{somesplgraphs}
\end{figure}

\section{Notations and Preliminaries}\label{2}
The vertex set, edge set, complement, \emph{clique number} and the \emph{independence number} of a graph $G$ are denoted by $V(G)$, $E(G)$, $\overline{G}$, $\omega(G)$ and $\alpha(G)$ respectively. When there is no ambiguity, $\omega(G)$ will be denoted by $\omega$.
For $T,S\subseteq V(G)$, let $N_S(T)=N(T)\cap S$ (where $N(T)$ denotes the set of all neighbors of $T$ in $G$), let $\langle T\rangle$ denote the subgraph induced by $T$ in $G$ and let $[T,S]$ denote the set of all edges in $G$ with one end in $T$ and the other end in $S$. 
If every vertex in $T$ is adjacent with every vertex in $S$, then $ [T, S ]$ is said to be \emph{complete}. 
For any positive integer $k$, let $[k]$ denote the set $\{1,2,\ldots k\}$.
The \emph{lexicographic ordering} on the set $L=\{(i, j): 1 \leq i < j \leq \omega\}$ is defined in the following way.
For two distinct elements $(i_1,j_1),(i_2,j_2)\in L$, we say that $(i_1,j_1)$ precedes $(i_2,j_2)$, denoted by $(i_1,j_1)<_L(i_2,j_2)$ if either $i_1<i_2$ or $i_1=i_2$ and $j_1<j_2$.
We define the following sets to partition our vertex set $V(G)$.\\
(1) $A=\{v_1,v_2,\ldots,v_{\omega}\}$, a maximum clique of $G$, \\
(2) For any $k\in [\omega]$,\\
\phantom{(2) }$I_k=\{v\in V(G)\backslash A: v\in N(v_i), \text{\ for\ every}\ i\in[\omega]\backslash\{k\}\}$ \\
\phantom{(2) $I_k$} $=\{v\in V(G)\backslash A: v\notin N(v_i) , \text{\ only\ for\ }\ i=k\}$,  and \\
(3) $C_{i,j}=\{v\in V(G)\backslash A:v\notin N(v_i)$ and $\ v\notin N(v_j)\}\backslash\left\{\mathop\cup\limits_{(i',j')<_L(i,j)} C_{i',j'}\right\}$\\
\phantom{(3) $C_{i,j} $}$=\{v\in V(G)\backslash A: i\ \&\ j\ \textnormal{are the least distinct integers}\ \textnormal{such that}\ v\notin N(v_i)$ and $\ v\notin N(v_j)  \}$  \\
Without much difficulty, one can observe the following.\\
(i) For every $k,l\in [\omega]$, $I_k$ is an independent set and $I_k\cap I_l=\emptyset$.\\
(ii) $C_{i,j}\cap C_{i',j'}=\emptyset$ whenever $(i,j)\neq (i',j')$.\\
(iii) $I_k\cap C_{i,j}=\emptyset$, for any $(i,j)\in L$ and $k\in [\omega]$\\
(iv) For every $(i,j)\in L$, if a vertex $a\in C_{i,j}$, then $N_A(a)\supseteq \{v_1,v_2,\ldots,v_j\}\backslash\{v_i,v_j\}$.\\
(v) $V(G)=A\cup\left(\mathop\cup\limits_{k=1}^{\omega}I_k\right)\cup\left(\mathop\cup\limits_{(i,j)\in L}C_{i,j}\right)$.

We can also write  $V(G)=V_1\cup V_2$, where $V_1=\mathop\cup\limits_{1\leq k\leq\omega}\Big((\{v_k\}\cup I_k)=U_k \textnormal{ (say)} \Big)$ and $V_2=\mathop\cup\limits_{(i,j)\in L}C_{i,j}$.
We can further partition $V_2$ using the same partitioning method as defined above and write $V_2= A'\cup\left(\mathop\cup\limits_{k=1}^{|A'|}I'_k\right)\cup\left(\mathop\cup\limits_{(i,j)\in L'}C'_{i,j}\right)$, where $\omega(\langle V_2\rangle)=|A'|$, $A'=\{v'_1,v'_2,\ldots,v'_{|A'|}\}$ and $L'=\{(i, j): 1 \leq i < j \leq |A'|\}$.
The other properties and definitions for $I'_k$'s, $C'_{i,j}$'s and $L'$ are same as that of $I_k$'s, $C_{i,j}$'s and $L$ respectively.
Let $X_1=V_1$, $X_2=A'\cup\left(\mathop\cup\limits_{k=1}^{|A'|}I'_k\right)$ and $X_3=\mathop\cup\limits_{(i,j)\in L'}C'_{i,j}$.
Throughout this paper, we shall use the above mentioned partitions $V(G)=V_1\cup V_2$ and $V(G)=X_1\cup X_2\cup X_3$.

\section{$\chi$-binding function for $\{P_3\cup 2K_1,\overline{P_3\cup 2K_1}\}$-free graphs}

We begin this section by  recalling some of the results in \cite{chudnovsky2006strong,caldam2022chromaticspringerversion,randerathvizing,randerath2004}.
Note that when $p=2$, $(K_1\cup K_2)+K_p\cong \overline{P_3\cup 2K_1}$.
\begin{proposition}\cite{caldam2022chromaticspringerversion}\label{k1k2kpprop}
Let $G$ be a $(\overline{P_3\cup 2K_1}$)-free graph with $\omega(G)\geq 4$.
Then $G$ satisfies the following.
\begin{enumerate}[(i)]
\setlength\itemsep{-1pt}
\item \label{completeIkIl} For $k,\ell\in [\omega]$, $[I_k,I_{\ell}]$ is complete. Thus $\langle V_1\rangle$ is a complete multipartite graph with  $U_{t}$'s as its partitions.

\item \label{emptycij} For $j\geq 4$ and $1\leq i<j$, $C_{i,j}=\emptyset$. As a consequence, $V_2=C_{1,2}\cup C_{1,3}\cup C_{2,3}$.

\item \label{p-1neighbors} For any $x\in V_2$, $x$ has neighbors in  at most one $U_\ell$, where $\ell\in [\omega]$.
\end{enumerate}
\end{proposition}
\begin{theorem}\cite{chudnovsky2006strong}\label{spgt}
A graph $G$ is perfect if and only if it does not contain $C_{2n+1}$ and $\overline{C_{2n+1}}$ ($n\geq 2$) as induced subgraphs.
\end{theorem}
\begin{theorem}\cite{randerathvizing}\label{chair}
If $G$ is a $\{chair,\overline{P_3\cup 2K_1}\}$-free graph, then $\chi(G)\leq \omega+1$.
\end{theorem}
\begin{theorem}\cite{randerath2004}\label{3color}
If $G$ is a triangle-free and $H$-free graph, then $\chi(G)\leq 3$.
\end{theorem}
For establishing an optimal $\chi$-binding function for $\{P_3\cup 2K_1,\overline{P_3\cup 2K_1}\}$-free graphs with $\omega\geq4$, we begin by observing Lemma \ref{lemma}, Lemma \ref{lemma2}, Theorem \ref{char} and Lemma \ref{lemma3}.
\begin{lemma}\label{lemma}
Let $G$ be a $\{P_3\cup 2K_1,\overline{P_3\cup 2K_1}\}$-free graph with $\omega(G)\geq 4$ and let $X$ be an independent subset of $V_2$ such that $|X|\geq 2$. Then we have the following.  
\begin{enumerate}[(i)]
\item\label{twou} $N_{V_1}(X)\subseteq U_p\cup U_q$, for some $p,q\in [\omega]$.
\item\label{oneu} If $N(X)\cap U_p\neq\emptyset$ and $N(X)\cap U_q\neq\emptyset$, for distinct $p,q\in [\omega]$, then $|X|=2$.
\item\label{twoI} There exist integers $r,s\in [\omega]$ such that $\mathop\cup\limits_{k=1}^{\omega}I_k\subseteq I_r\cup I_s$.
\end{enumerate}
\end{lemma}
\begin{proof}
Let $G$ be a $\{P_3\cup 2K_1,\overline{P_3\cup 2K_1}\}$-free graph with $\omega(G)\geq 4$ and let $X$ be an independent subset of $V_2$ such that $|X|\geq 2$.
\begin{enumerate}[(i)]
\item Suppose $N_{V_1}(X)\nsubseteq U_p\cup U_q$, for any $p,q\in [\omega(G)]$, then by (\ref{p-1neighbors}) of Proposition \ref{k1k2kpprop},  there exist three vertices, say $\{a,b,c\}\subseteq X$ such that each of them have neighbors in unique $U_{\ell}$'s, say $\{U_a,U_b,U_c\}$ respectively.
Since $\omega(G)\geq 4$, by (\ref{completeIkIl}) and (\ref{p-1neighbors}) of Proposition \ref{k1k2kpprop}, there exists a vertex $v_s\in A\backslash\{v_a,v_b,v_c\}$ such that $P_3\cup 2K_1\sqsubseteq\langle\{a,U_a,v_s,b,c\}\rangle$, a contradiction.
\item If $N(X)\cap U_p\neq\emptyset$ and $N(X)\cap U_q\neq\emptyset$, for distinct $p,q\in[\omega]$, then by (\ref{p-1neighbors}) of Proposition \ref{k1k2kpprop}, there exist distinct vertices $a,b\in X$ such that $N(a)\cap U_p\neq \emptyset$ and $N(b)\cap U_q\neq \emptyset$.
If $|X|\geq 3$, say $c\in X\backslash \{a,b\}$, then for the same reason either $N(c)\cap U_p=\emptyset$ or $N(c)\cap U_q=\emptyset$.
Without loss of generality, let us assume that $N(c)\cap U_p=\emptyset$.
Then again by (\ref{completeIkIl}) and (\ref{p-1neighbors}) of Proposition \ref{k1k2kpprop}, there exists a vertex $v_s\in A$ such that $P_3\cup 2K_1\sqsubseteq\langle\{a,U_p,v_s,b,c\}\rangle$, a contradiction.
\item For the sake of contradiction, let $t_1,t_2,t_3\in [\omega]$ such that $I_{t_1}\neq \emptyset$, $I_{t_2}\neq \emptyset$ and $I_{t_3}\neq \emptyset$. Let $a,b\in X$. By (\ref{p-1neighbors}) of Proposition \ref{k1k2kpprop},   $[\{a,b\}, U_{t_i}]=\emptyset$ for some $i\in\{1,2,3\}$ (say, $i=1$) and there exists a vertex $v_{s}\in A\backslash\{v_{t_1}\}$ such that $[\{a,b\}, v_s]=\emptyset$.
Hence, by (\ref{completeIkIl}) of Proposition \ref{k1k2kpprop}, $P_3\cup 2K_1\sqsubseteq \langle\{v_s,U_{t_1},a,b\}\rangle$, a contradiction.
\end{enumerate}
\vskip -10mm
\end{proof}

\begin{lemma}\label{lemma2}
Let $G$ be a $\{P_3\cup 2K_1,\overline{P_3\cup 2K_1}\}$-free graph with $\omega(G)\geq 4$. Then   
\begin{enumerate}[(i)]
\item\label{lem21} $\langle V_2\rangle$ is $(P_3\cup K_1)$-free.
\item\label{lem22} $\langle X_3= (C'_{1,2}\cup C'_{1,3}\cup C'_{2,3})\rangle$ is $P_3$-free, a union of cliques.
\item\label{lem23} If $\omega(\langle V_2\rangle)\geq 4$ and $X_3\neq \emptyset$, then there exists a unique $t\in [|A'|]$ such that $\mathop\cup\limits_{r=1}^{|A'|}I_r'\subseteq I_t'$.
\item\label{lem24} If $\omega(\langle X_3\rangle)=\omega(G)$ and $[X_3, X_2]\neq \emptyset$, then $\langle X_3 \rangle$ is a clique of size $\omega(G)$.
\item\label{lem25}If $\omega(\langle X_2\rangle)=\omega(\langle X_3\rangle)=\omega(G)$, then $\langle X_2\rangle$ is a complete multipartite graph and any vertex in $X_i$ is adjacent to exactly $0,1$ or $\omega-1$ partitions in $X_j$, whenever $i<j$.
\end{enumerate}
\end{lemma}
\begin{proof} Let $G$ be a $\{P_3\cup 2K_1,\overline{P_3\cup 2K_1}\}$-free graph with $\omega(G)\geq 4$.
\begin{enumerate}[(i)]
\item For the sake of contradiction, let us assume that there exists a $(P_3\cup K_1)\sqsubseteq \langle V_2\rangle$.
By (\ref{p-1neighbors}) of Proposition \ref{k1k2kpprop} and (\ref{twou}) of Lemma \ref{lemma}, there exists a vertex $v_s\in V_1$ such that $\langle V(P_3\cup K_1)\cup v_s\rangle\cong P_3\cup 2K_1$, a contradiction.

\item Suppose $\langle X_3\rangle$ contains a $P_3$ say $P$, then $\omega(\langle V_2\rangle)\geq 2$.
If $\omega(\langle V_2\rangle)=2$, then $X_3=C_{1,2}'$ and $\langle\{V(P),v'_1\}\rangle\cong P_3\cup K_1$, a contradiction to (i) of Lemma \ref{lemma2}.
Hence, let us assume that $\omega(\langle V_2\rangle)\geq 3$. 
We see that $\langle C'_{i,j}\rangle$ is $P_3$-free, for every $(i,j)\in\{(1,2),(1,3),(2,3)\}$.
We shall now show that $[C_{i,j}',C_{k,l}']$ is complete for any distinct pair $(i,j),(k,l)\in\{(1,2),(1,3),(2,3)\}$. 
On the contrary, let us assume that there exists a distinct pair $(i,j)$ and $(k,l)$, say $(i,j)<_{L'}(k,l)$ and two vertices $a\in C_{i,j}'$ and $b\in C_{k,l}'$ such that $ab\notin E(G)$.
Since $(i,j),(k,l)\in\{(1,2),(1,3),(2,3)\}$, one can see that $\{i,j\}\cap\{k,l\}\neq \emptyset$ and as a result $i=k$ or $j=k$ or $j=l$. In any case, $\langle\{b,v_j',v_i',a\}\rangle\cong P_3\cup K_1$, a contradiction to (i) of Lemma \ref{lemma2}.
Hence $V(P)\subseteq (C_{i,j}'\cup C_{k,l}')$, for a distinct pair $(i,j),(k,l)\in\{(1,2),(1,3),(2,3)\}$ such that $(i,j)<_{L'}(k,l)$.
Here also either $i=k$ or $j=k$ or $j=l$,  and hence $P_3\cup K_1\sqsubseteq\langle\{V(P),v_i',v_j'\}\rangle$, a contradiction.

\item Let $\omega(\langle V_2\rangle)\geq 4$ and $X_3\neq \emptyset$, say $a\in X_3$. Suppose $\mathop\cup\limits_{r=1}^{|A'|}I_r'\nsubseteq I_t'$, then there exist two integers $t_1,t_2\in [|A'|]$ such that $I_{t_1}'\neq \emptyset$ and $I_{t_2}'\neq \emptyset$.
By using (\ref{p-1neighbors}) of Proposition \ref{k1k2kpprop}, we see that there exists a vertex $v_{s}'\in A'\backslash \{v_{t_1}',v_{t_2}'\}$ such that $av'_s\notin E(G)$ and $[a,U'_{k}]=\emptyset$ for at least one $k\in \{t_1,t_2\}$, say $t_1$.
Clearly by (\ref{completeIkIl}) of Proposition \ref{k1k2kpprop}, $P_3\cup K_1\sqsubseteq \langle\{U_{t_1}',v'_s, a\}\rangle$, a contradiction.

\item Let $\omega(\langle X_3\rangle)=\omega(G)$ and $[X_3, X_2]\neq \emptyset$. If $\langle X_3\rangle$ is not a clique,  then by (\ref{lem22})  of Lemma \ref{lemma2}, $\langle X_3\rangle$ will be a union of cliques. Let $H\subseteq X_3$ such that $\langle H\rangle$ is a clique of size $\omega$.
If $[\{X_3\backslash H\}, X_2]\neq \emptyset$ then there exists $a\in X_3\backslash H$ and $b\in X_2$ such that $ab\in E(G)$.
Since $\omega(\langle H\rangle)=\omega(G)$, there exists a vertex, say $h_1\in H$ such that $bh_1\notin E(G)$.
By (\ref{completeIkIl}) and (\ref{p-1neighbors}) of Proposition \ref{k1k2kpprop},  there exists a vertex $v_{p}'\in A'$ such that $\langle\{a,b,v_{p}',h_1\}\rangle\cong P_3\cup K_1$, a contradiction.
Hence $[\{X_3\backslash H\}, X_2]= \emptyset$.
Since $[X_3,X_2]\neq \emptyset$, $[H,X_2]\neq \emptyset$ and hence there exist $h_2 \in H$ and $b\in X_2$ such that $h_2b\in E(G)$. Since $H$ is a clique of size $\omega$, there exists a vertex $h_3\in H$ such that $bh_3\notin E(G)$. Therefore for any $a\in X_3\backslash H$,  $\langle\{h_3,h_2,b,a\}\rangle\cong P_3\cup K_1$, again a contradiction.

\item Let $\omega(\langle X_3\rangle)=\omega(\langle X_2\rangle)=\omega(G)$.
If $[X_3,X_2]\neq \emptyset$, then by (\ref{lem24}) of Lemma \ref{lemma2}, $\langle X_3\rangle$ is a clique and the proof is straightforward from the fact that $G$ is $(\overline{P_3\cup 2K_1}$)-free.
Hence let us assume that $[X_3,X_2]=\emptyset$.
By (\ref{lem22}) of Lemma \ref{lemma2}, $\langle X_3\rangle$ is a union of cliques and can be partitioned into $\omega$-independent sets. 
Let $T\subseteq V_2$ such that $\langle T\rangle$ is a clique of size $\omega$.
Since $[X_3,X_2]=\emptyset$, let  us assume that there exists a vertex $a\in X_1$ such that it is adjacent to more that $1$ and less that $\omega-1$ partitions in $X_j$. 
If $j=2$, then by (\ref{completeIkIl}) of Proposition \ref{k1k2kpprop}, $\overline{P_3\cup 2K_1}\sqsubseteq \langle X_2\cup a \rangle$, a contradiction.
Hence let us assume $j=3$. It is not difficult to see that, since $a$ is adjacent to more that $1$ and less that $\omega-1$ partitions in $X_3$, $[a,T]$ is not complete and it is adjacent to a vertex in $X_3\backslash T$. 
Let $b\in X_3\backslash T$, $t_i\in T$ and $v'_q\in X_2$ such that $ab\in E(G)$ and $at_i,bt_i,av'_q\notin E(G)$.
By (\ref{completeIkIl}) and (\ref{p-1neighbors}) of Proposition \ref{k1k2kpprop}, we can find a vertex $v_s\in A$ such  that $\langle\{b,a,v_s,t_i,v'_q\}\rangle\cong P_3\cup 2K_1$, a contradiction.
\end{enumerate}
\vskip-10mm
\end{proof}

A simple consequence of (i) of Proposition \ref{k1k2kpprop} and (ii) of Lemma \ref{lemma2} is Theorem \ref{char}.

\begin{theorem}\label{char}
If $G$ is a $\{P_3\cup 2K_1,\overline{P_3\cup 2K_1}\}$-free graph with $\omega(G)\geq 4$, then 
$V(G)=X_1\cup X_2\cup X_3$ where $\langle X_1\rangle$ is a complete multipartite graph, $\langle X_2\rangle$ is a multipartite graph (complete multipartite if $\omega(\langle X_2\rangle)\geq 4$) and $\langle X_3\rangle$ is a union of cliques.
\end{theorem}

Let us recall that, by using (\ref{completeIkIl}) of Proposition \ref{k1k2kpprop}, $\langle X_1\rangle$ and $\langle X_2\rangle$ (if $\omega(\langle V_2\rangle)\geq 4$) are complete multipartite graphs and by using (\ref{p-1neighbors}) of Proposition \ref{k1k2kpprop}, each vertex in $X_3$ will have neighbors in at most one $U_i'$ in $X_2$ if $\omega(\langle V_2\rangle)\geq 4$ and at most one $U_j$ in $X_1$. For the same reason, each vertex in $X_2$ will have neighbors in at most one $U_k$ in $X_1$. To avoid repetition, we shall be using these facts without mentioning the reason.

\begin{lemma}\label{lemma3}
Let $G$ be a $\{P_3\cup 2K_1,\overline{P_3\cup 2K_1}\}$-free graph with $\omega(G)\geq 4$,  $\omega(\langle X_2\rangle)=\omega(\langle X_3\rangle)=\omega(G)$, $[X_2,X_3]\neq\emptyset$ and for every $a\in X_3$, $[a,X_2]\neq \emptyset\neq[a,X_1]$. Here $\mathop\cup\limits_{r=1}^{|A'|}I_r'\subseteq I_t'$, for some $t\in [|A'|]$ and if $I_t'\neq\emptyset$,  then
\begin{enumerate}[(i)]
\item\label{lem31} $|U'_{t}|=2$, say $U'_{t}=\{z_1,z_2\}$ such that $|N_{X_3}(z_1)|=\omega-1$, $|N_{X_3}(z_2)|=1$, $(N_{X_3}(z_1)\cup N_{X_3}(z_2))=X_3$ and $(N_{X_3}(z_1)\cap N_{X_3}(z_2))=\emptyset$. As a result, $[X_3,X_2\backslash U'_{t}]=\emptyset$.
\item\label{lem32} $\mathop\cup\limits_{k=1}^{\omega} I_k\subseteq I_p$, for some $p\in[\omega]$. Furthermore, if $X\subseteq U_p$ and $|X|\geq 2$, then $[X,U'_{t}]\neq \emptyset$. 
\end{enumerate}

\noindent Further, if $V_2\subseteq N(U_p)$, then the following holds.

\begin{enumerate}[(i)]
 \setcounter{enumi}{2}
\item\label{lem33}If there exists a vertex $y\in U_p$ such that $[y,U'_{t}]$ is complete, then $N_{X_3}(y)\subseteq N_{X_3}(z_2)$.
Furthermore, if $[y, X_2\backslash\{a\}]$ is complete, where $a\in X_2\backslash U'_{t}$, then $N_{X_3}(y)=\emptyset$.
\item\label{lem34}There exists a unique vertex $y_1\in U_p$ such that $[y_1, X_2\backslash U'_{t}]$ is complete and $N_{X_2}(U_p\backslash \{y_1\})\subseteq U'_{t}$ such that every vertex in $U_p\backslash \{y_1\}$ has at least one neighbor in $U'_{t}$. In addition, $N_{X_3}(y_1)\neq\emptyset$.
\item\label{lem36}$|U_p|\leq 3$. If $|U_p|=3$ and $U_p=\{y_1,y_2,y_3\}$ then $N_{X_3}(y_2)=N_{X_3}(z_1)$, $N_{X_3}(y_1)=N_{X_3}(y_3)=N_{X_3}(z_2)$, $z_1y_3,z_2y_2\in E(G)$ and $z_2y_2\notin E(G)$.
\end{enumerate}
\end{lemma}
\begin{proof}  Let $G$ be a $\{P_3\cup 2K_1,\overline{P_3\cup 2K_1}\}$-free graph with $\omega(G)\geq 4$,  $\omega(\langle X_2\rangle)=\omega(\langle X_3\rangle)=\omega(G)$, $[X_2,X_3]\neq\emptyset$ and for every $a\in X_3$, $[a,X_2]\neq \emptyset\neq[a,X_1]$. By (iii) and (iv) of Lemma \ref{lemma2}, $\mathop\cup\limits_{r=1}^{|A'|}I_r'\subseteq I_t'$, for some $t\in [|A'|]$ and $X_3$ is a clique, say $X_3=\{h_1,h_2,\ldots,h_\omega\}$.  In addition, let us assume that $I_t'\neq\emptyset$.
\begin{enumerate}[(i)]
\item We shall first observe that $N_{X_2}(X_3)\subseteq U'_{t}$.
Suppose there exists a vertex $a\in X_3$ with $[a,U'_{t}]=\emptyset$, then we can find a vertex $v'_{r}\in A'$ such that $P_3\cup K_1\sqsubseteq \langle \{U'_{t},v'_{r},a\}\rangle$, a contradiction.
Next, we can also observe that there exists a unique vertex $z_1\in U'_{t}$ such that $|[z_1,X_3]|=\omega-1$.
On the contrary, if $|[y,X_3]|\leq\omega-2$ for every vertex  $y\in U'_{t}$, then by (v) of Lemma \ref{lemma2}, $|[y,X_3]|\leq 1$, for every $y\in U'_{t}$. 
Since every vertex in $X_3$ has a neighbor in $X_2$ and more precisely in $U'_{t}$, there exist vertices $h_1,h_2,h_3\in X_3$ and $y_1,y_2\in U'_{t}$ such that $y_1h_1,y_2h_2\in E(G)$ and $[h_3,\{y_1,y_2\}]=\emptyset$.
Here also, we can find a vertex $v'_{r}\in A'\backslash \{v'_t\}$ such that $\langle\{y_1,v'_{r},y_2,h_3\}\rangle\cong P_3\cup K_1$, a contradiction.
Hence there exists a vertex $z_1\in U'_{t}$ such that $|[z_1,X_3]|=\omega-1$. Let $N_{X_3}(z_1)=\{h_1,h_2,\ldots,h_{\omega-1}\}$. 
For the uniqueness, suppose there exists another vertex $y_0\in U'_{t}\backslash \{z_1\}$ such that $|[y_0,X_3]|=\omega-1$, then we can see that when $N_{X_3}(z_1)\neq N_{X_3}(y_{0})$, $\overline{P_3\cup 2K_1}\sqsubseteq \langle\{y_0,N_{X_3}(z_1),z_1\}\rangle$  and when  $N_{X_3}(z_1)=N_{X_3}(y_0)$,  then we can find a vertex $v'_{r}\in A'\backslash \{v'_t\}$ such that $\langle\{z_1,v'_{r},y_0,h_\omega\}\rangle\cong P_3\cup K_1$, a contradiction. 
Hence $z_1$ is the unique vertex in $U'_{t}$ such that $|[z_1,X_3]|=\omega-1$. 
Similarly, one can also see that $N_{X_3}(U'_{t}\backslash \{z_1\})\cap N_{X_3}(z_1)=\emptyset$, which implies that $N_{X_3}(U'_{t}\backslash \{z_1\})\subseteq \{h_\omega\}$.
As a consequence, we can see that $|U'_{t}|=2$, say $U'_{t}=\{z_1,z_2\}$ and $z_2h_\omega\in E(G)$.  
Hence $N_{X_2}(X_3)=U'_{t}$ and $[X_3,X_2\backslash U'_{t}]=\emptyset$.
\item By (\ref{twoI}) of Lemma \ref{lemma}, we know that $\mathop\cup\limits_{k=1}^{\omega} I_k\subseteq I_p\cup I_q$, for some $p,q\in[\omega]$. Suppose, $I_p\neq \emptyset$ and $I_q\neq\emptyset$, it is not difficult to see that $[U'_{t},U_p]\neq\emptyset$ and $[U'_{t},U_q]\neq\emptyset$.
Without loss of generality, let us assume that $z_1\in N(U_p)$ and $z_2\in N(U_q)$.
Since $[z_1,h_{\omega}]=\emptyset$ and $[z_2,\mathop\cup\limits_{r=1}^{\omega-1} h_r]=\emptyset$, we see that $N_{X_1}(h_{\omega})\subseteq U_q$ and $N_{X_1}(\mathop\cup\limits_{r=1}^{\omega-1} h_r)\subseteq U_p$ respectively.
Hence every vertex $y\in X_2\backslash U'_{t}$ has neighbors in both $U_p$ and $U_q$ (otherwise, we can find a vertex in $v_r\in V_1\backslash \{U_p,U_q\}$, such that $P_3\cup 2K_1\sqsubseteq \{U_p\cup U_q,v_r,y,X_3\}$), a contradiction.
Furthermore, if $X\subseteq U_p$ and $|X|\geq 2$, then it is not difficult to see that $[X,U'_{t}]\neq \emptyset$.
\item For the sake of contradiction, let us assume that there exists a vertex $y\in U_p$ such that $[y,U'_{t}]$ is complete but $N_{X_3}(y)\nsubseteq N_{X_3}(z_2)$.
Let us first consider the case when  $|N_{X_3}(y)|=\omega-1$.
If $N_{X_3}(y)=N_{X_3}(z_1)$, then $\langle\{h_{\omega},h_1,h_2,y,z_1\}\rangle\cong \overline{P_3\cup 2K_1}$, a contradiction and if $N_{X_3}(y)\neq N_{X_3}(z_1)$, then $h_{\omega},h_i,h_j\in N_{X_3}(y)$, for some $i,j\in [\omega-1]$ and thereby  $\langle\{z_2,y,h_{\omega},h_i,h_j\}\rangle\cong \overline{P_3\cup 2K_1}$, again a contradiction.
Hence $|N_{X_3}(y)|\leq 1$.
Clearly, if $N_{X_3}(y)\cap N_{X_3}(z_1)\neq \emptyset$, say $h_1\in N_{X_3}(y)\cap N_{X_3}(z_1)$, then $\langle\{y,z_1,h_1,h_2,h_3\}\rangle\cong \overline{P_3\cup 2K_1}$, a contradiction.
Therefore $N_{X_3}(y)\subseteq N_{X_3}(z_2)$.
Furthermore, if $[y, X_2\backslash\{a\}]$ is complete, then $N_{X_3}(y)\neq\emptyset$ will yield $\overline{P_3\cup 2K_1}\sqsubseteq \langle\{h_{\omega},y,z_2,X_2\backslash U'_{t}\}\rangle$, a contradiction.

\item We begin by showing that there exists a vertex $y\in U_p$ such that $|[y, X_2]|\geq \omega-1$.
On the contrary, let us assume that $|[y, X_2]|\leq \omega-2$, for every $y\in U_p$.
Then by (\ref{lem25}) of Lemma \ref{lemma2}, every $y\in U_p$, $y$ has neighbors in at most one partition in $X_2$.
In addition, since $V_2\subset N(U_p)$, we can find vertices $x_1,x_2\in U_p$ such that $[\{x_1,x_2\},U'_t\}]=\emptyset$, a contradiction to (ii) of Lemma \ref{lemma3}.
Hence there exists a vertex say $y_1\in U_p$ such that $|[y_1, X_2]|\geq \omega-1$.
If $|[y_1, X_2]|\geq \omega-1$ and $[y_1, U'_{t}]$ is neither $\emptyset$ nor complete, then one can observe that $\overline{P_3\cup 2K_1}\sqsubseteq \langle\{y_1\cup X_2\}\rangle$, a contradiction.
Thus if $|[y_1,U_p]|\geq \omega-1$, then either $[y_1, X_2\backslash U'_{t}]$ is complete or $[y_1, X_2\backslash\{a\}]$ is complete, for some $a\in X_2\backslash U'_{t}$. Let us next show that such a vertex $y_1$ is unique. 
Suppose $y_1$ is not unique and there exists another vertex $y_0\in U_{p}\backslash\{y_1\}$ such that $|[y_0, X_2]|\geq \omega-1$.
If $N_{X_2}(y_1)\neq N_{X_2}(y_0)$, then $\overline{P_3\cup 2K_1}\sqsubseteq\langle\{y_1,N_{X_2}(y_0),y_0\}\rangle$, a contradiction.
Hence $N_{X_2}(y_1)=N_{X_2}(y_0)$.
If $N_{X_2}(y)=N_{X_2}(y_0)=X_2\backslash U'_{t}$, then $[\{y_1,y_0\},U'_{t}]=\emptyset$, a contradiction to (ii) of Lemma \ref{lemma3} and if $N_{X_2}(y_1)=N_{X_2}(y_0)=X_2\backslash \{a\}$, for some $a\in X_2\backslash U'_{t}$, then by (i) and (iii) of Lemma \ref{lemma3}, we can find a vertex $v_s\in A\backslash\{v_p\}$ such that $\langle\{y_1,v_s,y_0,h_1,a\}\rangle\cong P_3\cup 2K_1$, again a contradiction.
Thus $y_1$ is the only vertex in $U_p$ such that $|[y_1,X_2]|\geq \omega-1$ and as a result either $[y_1, X_2\backslash U'_{t}]$ is complete or $[y_1, X_2\backslash\{a\}]$ is complete, where $a\in X_2\backslash U'_{t}$.
We shall now show that $[y_1, X_2\backslash U'_{t}]$ being complete is the only possibility. 
Suppose $[y_1, X_2\backslash\{a\}]$ is complete, then by (iii) of Lemma \ref{lemma3}, $N_{X_3}(y_1)=\emptyset$.
Since $X_3\subseteq N(U_p)$, $|U_p|\geq 3$.
Let $y_2,y_3\in U_{p}\backslash \{y_1\}$.
We shall show that $y_2a,y_3a\in E(G)$.
On the contrary, let us assume that $y_2a\notin E(G)$, then there exist vertices $v_s\in V_1$ and $h_{\ell}\in X_3$ (since $N_{X_3}(y_1)=\emptyset$ and $|N_{X_3}(y_2)|\leq \omega-1$), for some $\ell,s\in [\omega]$ such that $\langle\{y_1,v_s,y_2,h_{\ell},a\}\rangle\cong P_3\cup 2K_1$, a contradiction.
Similarly, one can show that $y_3a\in E(G)$ and hence $[\{y_2,y_3\},U'_t]=\emptyset$, again a contradiction to (ii) of Lemma \ref{lemma3}.
Hence $[y_1,X_2\backslash U'_{t}]$ being complete is the only possibility.  For similar reasons $N_{X_2}(U_p\backslash \{y_1\})\subseteq U'_{t}$ and every vertex in $U_p\backslash \{y_1\}$ has at least one neighbor in $U'_{t}$.
Finally, let us show that $N_{X_3}(y_1)\neq \emptyset$.
For the sake of contradiction, let us assume that $N_{X_3}(y_1)=\emptyset$. 
Since $X_3\subseteq N(U_p)$, $|U_p|\geq 3$.
Let us consider $y_2,y_3\in U_{p}\backslash \{y_1\}$ such that $U'_{t}\subseteq N(y_2)\cup N(y_3)$ (since $X_2\subseteq N(U_p)$).
If $X_3\nsubseteq \{N(y_2)\cup N(y_3)\}$, then by (\ref{lem31}) of Lemma \ref{lemma3} and the fact that $N_{X_2}(U_p\backslash \{y_1\})\subseteq U'_t$, we can find vertices $v_r'\in X_2\backslash U'_{t}$, $v_s\in V_1$ and $h_i\in X_3\backslash \{N(y_2)\cup N(y_3)\}$  such that $\langle\{y_2,v_s,y_3,v'_r,h_i\}\rangle\cong P_3\cup 2K_1$, a contradiction.
Hence $(X_3\cup U'_{t})\subseteq \{N(y_2)\cup N(y_3)\}$.
Clearly, by (\ref{lem25}) of Lemma \ref{lemma2}, one of the vertices in $\{y_2,y_3\}$ has $\omega-1$ neighbors in $X_3$, say $y_2$ and by (\ref{lem33}) of Lemma \ref{lemma3}, $[y_2,U'_{t}]$ is not complete.
If $z_1y_2\in E(G)$ and $z_2y_2\notin E(G)$, then it is not difficult to see that $N_{X_3}(z_1)\neq N_{X_3}(y_2)$ (if $N_{X_3}(z_1)=N_{X_3}(y_2)$, then $\omega(\langle\{z_1,y_2,\mathop\cup\limits_{r=1}^{\omega-1}h_r\}\rangle)=\omega+1$, a contradiction).
Hence let us assume that $N_{X_3}(y_2)=\mathop\cup\limits_{r=2}^{\omega}h_r$.
Since $(X_3\cup U'_{t})\subseteq \{N(y_2)\cup N(y_3)\}$, $y_3z_2,y_3h_1\in E(G)$.
As a result, by (\ref{lem33}) of Lemma \ref{lemma3}, $y_3z_1\notin E(G)$ and thus $\langle\{z_2,y_3,h_1,y_1,y_2\}\rangle\cong P_3\cup 2K_1$, a contradiction.
Now, let us consider the case when $z_2y_2\in E(G)$ and $z_1y_2\notin E(G)$.
Since $|N_{X_3}(y_2)|=\omega-1$, one can observe that $N_{X_3}(y_2)=N_{X_3}(z_1)$ (else, we get an induced $\overline{P_3\cup 2K_1}$).
Hence we can find a vertex $v_s\in X_1$ such that $\langle\{y_1,v_s,y_2,z_1,h_{\omega}\}\rangle\cong P_3\cup 2K_1$, a contradiction.
Therefore, $N_{X_3}(y_1)\neq \emptyset$.
\item Let us start by showing that $|U_p|\leq 3$.
By (\ref{lem34}) of Lemma \ref{lemma3}, there exists a $y_1\in U_p$ such that $X_2\backslash U'_{t}\subset N(y_1)$ and $[X_2\backslash U'_t, U_{p}\backslash\{y_1\}]=\emptyset$.
Suppose $|U_p|\geq 4$, then there exists at least three vertices, say, $y_2,y_3,y_4\in U_p\backslash\{y_1\}$.
By similar arguments as in (\ref{lem34}) of Lemma \ref{lemma3}, one can observe that $X_3$ is dominated by any two distinct vertices in $\{y_2,y_3,y_4\}$ and the number of neigbors of any. Let us first consider the case when $X_3$ is dominated by  $y_2$ and $y_3$.  
As $X_3\subset \{N(y_2)\cup N(y_3)\}$, one can see that $N_{X_3}(y_2)\neq N_{X_3}(y_3)$ and if $|N_{X_3}(y_2)|=|N_{X_3}(y_3)|=\omega-1$ then  we get an induced $\overline{P_3\cup 2K_1}$, a contradiction. Hence  without loss of generality, let us assume that $|N_{X_3}(y_2)|=\omega-1$, $|N_{X_3}(y_3)|=1$ and $N_{X_3}(y_2)\cap N_{X_3}(y_3)=\emptyset$.
Next, let us consider $y_2,y_4\in U_{p}\backslash\{y_1\}$.
Since $|N_{X_3}(y_2)|=\omega-1$, again by using similar arguments, one can see that $|N_{X_3}(y_4)|=1$ such that  $N_{X_3}(y_2)\cap N_{X_3}(y_3)=\emptyset$ and as a consequence, we get that $X_3\nsubseteq N(\{y_3,y_4\})$, a contradiction.
Thus $|U_p|\leq 3$.
Let us now consider the case when $|U_p|=3$, say,
$U_p=\{y_1,y_2,y_3\}$, where $y_1$ is as defined in (\ref{lem34}) of Lemma \ref{lemma3}.
For the same reason as mentioned earlier, without loss of generality, let us assume that $|N_{X_3}(y_2)|=\omega-1$, $|N_{X_3}(y_3)|=1$ and $N_{X_3}(y_2)\cap N_{X_3}(y_3)=\emptyset$.
Also by (\ref{lem34}) of Lemma \ref{lemma3}, we know that $N_{X_3}(y_1)\neq \emptyset$.
Hence by (\ref{lem25}) of Lemma \ref{lemma2}, $|N_{X_3}(y_1)|=\omega-1 \textnormal{ or }1$.
Let us divide the proof depending on $|N_{X_3}(y_1)|$. 

\textbf{Case 1} $|N_{X_3}(y_1)|=\omega-1$

One can see that since $z_1y_1\notin E(G)$, $N_{X_3}(y_1)=N_{X_3}(z_1)=\mathop\cup\limits_{r=1}^{\omega-1}h_r$  (otherwise $\overline{P_3\cup 2K_1}\sqsubseteq \langle\{y_1,X_3,z_1\}\rangle$, a contradiction).
Similarly, since $|N_{X_3}(y_2)|=\omega-1$ and $y_1y_2\notin E(G)$, we get that $N_{X_3}(z_1)=N_{X_3}(y_1)=N_{X_3}(y_2)$ $\implies$ $y_2z_1\notin E(G)$ (else, $\omega(\langle\{z_1,y_2,\mathop\cup\limits_{r=1}^{\omega-1}h_r\}\rangle)=\omega+1$) and as a consequence we see that  $y_2z_2,y_3z_1,y_3h_{\omega}\in E(G)$ (by (\ref{lem34}) of Lemma\ref{lemma3} and the fact that $(X_3\cup U'_{t})\subset \{N(y_2)\cup N(y_3)$\}).
Thus we can find a $v_s\in A\backslash\{v_p\}$ such that $\langle\{y_1,v_s,y_2,z_1,h_{\omega}\}\rangle\cong P_3\cup 2K_1$, a contradiction and hence case 1 is not possible.

\textbf{Case 2} $|N_{X_3}(y_1)|=1$ and $N_{X_3}(y_2)\neq N_{X_3}(z_1)$.

Since $N_{X_3}(y_2)\neq N_{X_3}(z_1)$, $z_1y_2\in E(G)$ (otherwise we will get an induced $\overline{P_3\cup 2K_1}$).  Without loss of generality, let us assume that $y_2h_1\notin E(G)$.
 Since $X_3\subset N(U_p)$, one can see that $y_3h_1\in E(G)$. By using (\ref{lem33}) and (\ref{lem34}) of Lemma \ref{lemma3}, $z_2y_3\in E(G)$ and $z_2y_2, z_1y_3\notin E(G)$.
Here if $y_1h_1\notin E(G)$, then $\langle\{h_1,y_3,z_2,y_1,y_2\}\rangle\cong P_3\cup 2K_1$ and if $y_1h_1\in E(G)$, then we can find a vertex $v_s\in A\backslash \{v_p\}$ such that $\langle\{y_1,v_s,y_3,z_1,h_{\omega}\}\rangle\cong P_3\cup 2K_1$, in any case a contradiction and hence case 2 is also not possible.

\textbf{Case 3} $|N_{X_3}(y_1)|=1$ and $N_{X_3}(y_2)= N_{X_3}(z_1)$.

Here $N_{X_3}(y_2)=N_{X_3}(z_1)=\mathop\cup\limits_{r=1}^{\omega-1}h_r$ and $N_{X_3}(y_3)=N_{X_3}(z_2)=\{h_{\omega}\}$.
By using similar arguments as in Case 1 and Case 2, one can observe that $y_2z_1,y_2h_{\omega}\notin E(G)$ and as a consequence $y_2z_2,y_3z_1,y_3h_{\omega}\in E(G)$ which in turn implies that $y_1h_{\omega}\in E(G)$. 
\end{enumerate}
\vskip-10mm
\end{proof}
Let us now determine an optimal $\chi$-binding function for $\{P_3\cup 2K_1,\overline{P_3\cup 2K_1}\}$-free graph with $\omega(G)\geq 4$.  
\begin{theorem}\label{main}
If $G$ is a $\{P_3\cup 2K_1,\overline{P_3\cup 2K_1}\}$-free graph with $\omega(G)\geq 4$, then $\chi(G)\leq \omega(G)+1$.
\end{theorem}
\begin{proof}
Let $G$ be a $\{P_3\cup 2K_1,\overline{P_3\cup 2K_1}\}$-free graph with $\omega(G)\geq 4$. We know that $V(G)=V_1\cup V_2=X_1\cup X_2\cup X_3$.
By using (\ref{emptycij}) of Proposition \ref{k1k2kpprop}, one can see that $V_2=C_{1,2}\cup C_{1,3}\cup C_{2,3}$ and $X_3=C'_{1,2}\cup C'_{1,3}\cup C'_{2,3}$.

Let $\{1,2,\ldots,\omega(G)+1\}$ be the set of colors. Throughout the proof, when we say that `we assign a color to a set' or `there is a color available for a set', it means that we will assign the same color to all the vertices in that set. To establish a $(\omega+1)$-coloring for $G$, we begin by assigning the color $k$, to the vertices in $U_k$, for each $k\in[\omega]$. We shall divide the proof into cases depending upon $\omega(\langle V_2\rangle)$.

\noindent\textbf{Case 1} $\omega(\langle V_2\rangle)\leq \omega-1$.

By (i) of Lemma \ref{lemma2}, $\langle V_2\rangle$ is $\{P_3\cup K_1,\overline{P_3\cup 2K_1}\}$-free , and hence by Theorem \ref{chair}, we can partition $V_2$ into $\omega$ independent sets say, $S_1,S_2,\ldots S_{\omega}$.
By (\ref{twou}) of Lemma \ref{lemma}, we see that each $S_{\ell}$, $\ell\in [\omega-1]$ can be assigned a unique color from $[\omega+1]$.
Finally in $S_{\omega}$, each vertex is adjacent to at most $(\omega-1)$ colors in $V_2$ and one color in $V_1$.
Hence $G$ is $(\omega+1)$-colorable. 

\noindent\textbf{Case 2} $\omega(\langle V_2\rangle)= \omega\geq 4$ and $\mathop\cup\limits_{r=1}^{|A'|}I_r'\nsubseteq I_t'$, for any $t\in [|A'|]$. 

Since $\omega(\langle V_2\rangle)= \omega\geq 4$, $\langle X_2\rangle$ is a complete multipartite graph and since $\mathop\cup\limits_{r=1}^{|A'|}I_r'\nsubseteq I_t'$, for any $t\in [|A'|]$, by (iii) of Lemma \ref{lemma2},   $X_3=\emptyset$. Therefore, $\langle V_2\rangle$ is a complete multipartite graph with $\omega$-partites and as in Case 1,  $G$ will be $(\omega+1)$-colorable.

\noindent\textbf{Case 3} $\omega(\langle V_2\rangle)= \omega\geq 4$ and $\mathop\cup\limits_{r=1}^{|A'|}I_r'\subseteq I_t'$, for some $t\in [|A'|]$.

Here also, $\langle X_2\rangle$ is a complete multipartite graph. Let us further subdivide this depending on $\omega(\langle X_3\rangle)$.

\noindent\textbf{Case 3.1} $\omega(\langle X_3\rangle)\leq \omega-1$
 
Let us assign the color $(\omega+1)$ to $U'_{t}$ and assign a distinct color from $[\omega]$ to each of the vertex in $X_2\backslash U'_{t}$. Since $\omega(\langle X_3\rangle)\leq \omega-1$, by (iii) of Proposition \ref{k1k2kpprop}, for every vertex in $X_3$, there exists a proper color from $[\omega+1]$  and hence $G$ is $(\omega+1)$-colorable.

\noindent\textbf{Case 3.2} $\omega(\langle X_3\rangle)=\omega$

If $[X_3,X_2]=\emptyset$, then we can color $X_2$ as in Case 3.1 and since each vertex in $X_3$ is adjacent to at most $\omega-1$ colors in $X_3$, at most one color in $X_1$ and no vertex in $X_2$, each vertex in $X_3$ has an available color in $[\omega+1]$  and thus $G$ is $(\omega+1)$-colorable. Hence let us assume that $[X_3,X_2]\neq \emptyset$.  Then by (iv) of Lemma \ref{lemma2}, $\langle X_3 \rangle$ is a clique.
If there exists a vertex $a\in X_3$ such that $[a,X_2]=\emptyset$ or $[a,X_1]=\emptyset$, then without much difficulty, we can give a $(\omega+1)$- coloring to $G$, by coloring the vertices in $(X_2\cup X_3)$ in such a way that $a$ is the last vertex that is colored. 
Likewise, if $I'_{t}=\emptyset$, then $\langle X_3\rangle$ and $\langle X_2\rangle$ are cliques of size $\omega$ and hence as in Case 1, 
we can show that $G$ is $(\omega+1)$-colorable.

Hence the remaining case to be considered has $\omega(\langle X_2\rangle)=\omega(\langle X_3\rangle)=\omega(G)$,   $\mathop\cup\limits_{r=1}^{|A'|}I_r'\subseteq I_t'$, for some $t\in [|A'|]$ with $I_t'\neq\emptyset$, $[X_2,X_3]\neq\emptyset$ and for every $a\in X_3$, $[a,X_2]\neq \emptyset\neq[a,X_1]$.

Here,  $\langle X_3\rangle$ will be a clique.
Let $X_3=\{h_1,h_2,\ldots h_{\omega}\}$. By (i) of Lemma \ref{lemma3}, $|U'_{t}|=2$, say $U'_{t}=\{z_1,z_2\}$, where $N_{X_3}(z_1)=\mathop\cup\limits_{r=1}^{\omega-1} h_r$ and $N_{X_3}(z_2)=h_\omega$.
As a result, $\alpha(\langle V_2\rangle)=2$ and $|V_2|=2\omega+1$. Hence by (\ref{lem21}) of Lemma \ref{lemma2} and Theorem \ref{chair}, $\chi(\langle V_2\rangle)=\omega(G)+1$. 
Next, by (ii) of Lemma \ref{lemma3}, $\mathop\cup\limits_{k=1}^{\omega} I_k\subseteq I_p$, for some $p\in[\omega]$. Let us further divide this case into two more cases namely, $V_2\nsubseteq N(U_p)$ or $V_2\subseteq N(U_p)$.

\noindent\textbf{Case 3.2.1} $V_2\nsubseteq N(U_p)$

Let us begin by considering the case when $X_3\nsubseteq N(U_p)$. Without loss of generality, let us assume that $h_i\notin N(U_p)$, $i\in [\omega]$.
Since $[a,X_1]\neq\emptyset$, for every $a\in X_3$ and $\omega(\langle X_3\rangle)=\omega$, there exist vertices $v_s\in V_1$ and $h_j\in X_3$, $j\in [\omega]\backslash\{i\}$ such that $h_iv_s\in E(G)$ and $h_jv_s\notin E(G)$.
To establish an $(\omega+1)$-coloring, we color $X_2$ as done in Case 3.1. 
By (i) of Lemma \ref{lemma3}, $[X_3,X_2\backslash U'_{t}]=\emptyset$ and hence we can assign the color $s$ to $h_{j}$
 and we can properly color the vertices in $X_3\backslash\{h_i,h_j\}$. 
Finally for $h_i$, it is adjacent to the color $\omega+1$ (in $X_2$), the color $s$ in $X_1$ and $\omega-1$ colors in $X_3$ which includes the color $s$. Hence $G$ is $(\omega+1)$-colorable.
Next let us assume that $X_3\subseteq N(U_p)$. Then $X_2\nsubseteq N(U_p)$ and $|U_p|\geq 2$ (by (\ref{lem25}) of Lemma \ref{lemma2}).
From (ii) of Lemma \ref{lemma3}, we see that $[U_p,U'_{t}]\neq \emptyset$. Without loss of generality,  let $z_2\in N(U_p)$. Let us complete the coloring depending upon where the non-neighbors of $U_p$ are in  $X_2$.

\noindent\textbf{Case 3.2.1.1} $X_2\backslash U'_{t}\nsubseteq N(U_p)$

Let $a\in X_2\backslash U'_{t}$ such that $a\notin N(U_p)$. Define $S_1=\{z_1,h_\omega\}$, $S_\omega=\{a\}$ and $S_{\omega+1}=\{z_2,h_1\}$. Clearly $S_1, S_{\omega}, S_{\omega+1}$ are independents.  Since $[X_3,X_2\backslash U_t']=\emptyset$, we can partition the remaining vertices of $X_2\cup X_3$ into $\omega-2$ independent sets say $S_2, S_3,\ldots, S_{\omega-1}$. 
Since $X_3\subseteq N(U_{p})$, each $S_{\ell}$, $\ell\in [\omega-1]$ is adjacent to at most one color from $[\omega+1]\backslash\{p\}$ in $X_1$.
Hence we can assign a unique color from $[\omega+1]\backslash\{p\}$ to each $S_{\ell}, \ell\in [\omega-1]$.
Next we can assign the color $p$ to $S_{\omega}$.
Finally, for $S_{\omega+1}$, since $S_{\omega+1}\subseteq N(U_p)$ and the vertices in $S_{\omega+1}$ are adjacent to at most $\omega-1$  colors (other than $p$) in $V_2$, there is a color available from $[\omega+1]$ for $S_{\omega+1}$.
Hence $G$ is $(\omega+1)$-colorable.

\noindent\textbf{Case 3.2.1.2} $X_2\backslash U'_{t}\subseteq N(U_p)$

Since $X_2\backslash U'_{t}\subseteq N(U_p)$, $z_2\in N(U_p)$ and $X_2\nsubseteq U_p$, we get that $z_1\notin N(U_p)$.
Let us begin the $(\omega+1)$-coloring by coloring $z_1$ with $p$ and  $\{z_2,h_1\}$ with color $\omega+1$.
Since each vertex of $X_2\backslash U'_{t}$ is adjacent only to color $p$ in $X_1$ and $\omega$ colors in $X_2$ which includes $p$, we can properly color the vertices in $X_2\backslash U'_{t}$.
Finally for the vertices in $X_3$, since each vertex in $X_3$ is adjacent to one color in $X_1$, one color in $X_2$ and at most  $\omega-1$ colors in $X_3$ and the color of $h_1$ and $z_2$ are same, we can get a proper coloring for the vertices in $X_3$ by finally coloring the vertex $h_\omega$. Hence $G$ is $(\omega+1)$-colorable.

\noindent\textbf{Case 3.2.2} $V_2\subseteq N(U_p)$

As mentioned earlier, $\chi(\langle V_2\rangle)=\omega+1$.
Since $V_2\subset N(U_p)$, we have only $\omega$-colors remaining in $[\omega+1]\backslash\{p\}$ to color the vertices of $X_2\cup X_3$.
Hence we have to reassign the colors of some vertices in $U_p$ with $\omega+1$.
Let us reassign the vertex $y_1$ with color $\omega+1$, where $y_1$ is as defined in (\ref{lem34}) of Lemma \ref{lemma3}.
By (\ref{lem36}) of Lemma \ref{lemma3}, we know that $|U_p|\leq 3$.
Hence to establish an $(\omega+1)$-coloring of $G$, we consider the cases $|U_p|=3$ and $|U_p|=2$ independently.
Let us start by considering $|U_p|=3$, where $U_{p}=\{y_1,y_2,y_3\}$ and its properties are as defined in (\ref{lem36}) of Lemma \ref{lemma3}. We begin by assigning the color $\omega+1$ to $z_2$ and the color $s$ to $z_1$, where $s\in [\omega+1]\backslash\{\omega+1,p\}$.
By using (\ref{lem34}) of Lemma \ref{lemma3}, each vertex in $X_2\backslash U'_{t}$ is adjacent to only the colors $s$ and $\omega+1$ in $(U'_{t}\cup X_1)$ and as a result there is a color available from $[\omega+1]$ for all the vertices in $X_2\backslash U'_{t}$.
Finally for $X_3$, we first assign the color $s$ to $h_{\omega}$. Since each vertex in $X_3\backslash\{h_{\omega}\}$ is adjacent to the color $s$  in $X_2$, the color $p$ in $X_1$ and at most $\omega-1$ colors in $X_3$ ( including the color $s$), there is a color available from $[\omega+1]$ for all the vertices in $X_3\backslash\{h_{\omega}\}$.
Next, let us consider the case when $|U_p|=2$ and let $U_{p}=\{y_1,y_2\}$.
Since $V_2\subset N(U_p)$, we get that $[y_2,U'_t]$ is complete and hence by (\ref{lem33}) of Lemma \ref{lemma3}, $y_2h_{\omega}\in E(G)$ and $[y_1,\mathop\cup\limits_{r=1}^{\omega-1}h_r]$ is complete.  To establish a $(\omega+1)$-coloring for $G$, let us begin by assigning the color $\omega+1$ to $\{z_1,z_2\}$.
One can see that there is a color available for each vertex in $X_2\backslash U'_{t}$.
Since the vertices in $X_3\backslash\{h_\omega\}$ are adjacent only to the color $\omega+1$ in $X_2$ and $X_1$, we can get a proper coloring for the vertices in $X_3$  by first coloring the vertex $h_{\omega}$ and then by coloring the vertices in $X_3\backslash\{h_{\omega}\}$. 
Hence $G$ is $(\omega+1)$-colorable.
\end{proof}

Since $(P_3\cup 2K_1)$ is an induced subgraph of the $H$ graph (see Figure \ref{somesplgraphs}), by Theorem \ref{3color} and Theorem \ref{main}, we get  Corollary \ref{final}.

\begin{corollary}\label{final}
If $G$ is a $\{P_3\cup 2K_1,\overline{P_3\cup 2K_1}\}$-free graph with $\omega(G)\neq 3$, then $\chi(G)\leq \omega+1$. 
\end{corollary}

To show that the bound in Corollary \ref{final} is tight, we can consider the following examples.
For $\omega=2$, one can see that $C_5$ is an example and for $\omega\geq 4$, we can consider the graph induced by $V_2$ in Case 3.2.1 of Theorem \ref{main}.
Let us call it $G^*$, where $V(G^*)=\{h_1,h_2,\ldots,h_\omega\}\cup \{z_1,z_2,y_1,y_2,\ldots,y_{\omega-1}\}$, with $\omega\geq 4$ and $E(G)= \{\{h_ih_j\} \cup \{y_ry_s\}\cup \{z_1y_r\}\cup \{z_2y_r\}\cup \{z_1h_s\}\cup \{z_2h_\omega\}$, where $1\leq i,j\leq \omega$ and $1\leq r,s\leq \omega-1\}$.
It is not difficult to see that $G^*$ is a $\{P_3\cup 2K_1,\overline{P_3\cup 2K_1}\}$-free graph, with $|V(G^*)|=2\omega+1$, $\alpha(G^*)=2$, $\omega(G^*)=\omega$ and hence $\chi(G^*)=\omega(G^*)+1$.

Finally, let us consider $\{P_3\cup 2K_1,\overline{P_3\cup 2K_1}\}$-free graphs with $\omega=3$. Here we can give an example to show that if $f$ is a $\chi$-binding function for the class of $\{P_3\cup 2K_1,\overline{P_3\cup 2K_1}\}$-free graphs then $f(3)\geq 6$.
For instance, consider $G^{**}$ to be a circulant graph on $17$ vertices such that all $17$ vertices form a cycle with all chords of length $1,2,4$ and $8$ (see, \textcolor{blue}{http://www.cut-the-knot.org/ arithmetic/combinatorics/Ramsey44.shtml}).
Note that $G^{**}$ is a $8$-regular graph with $\omega(G^{**})=\alpha(G^{**})=3$ and $\chi(G^{**})=6$ (see, \textcolor{blue}{https://houseofgraphs.org/graphs/30395}).
Since $\omega(G^{**})=\alpha(G^{**})=3$, $G^{**}$ is $\{P_3\cup 2K_1,\overline{P_3\cup 2K_1}\}$-free with $\chi(G^{**})=6$. In theorem \ref{3case}, we shall show that the optimal $\chi$-binding function $f^*$ for the class of $\{P_3\cup 2K_1,\overline{P_3\cup 2K_1}\}$-free graphs satisfies $f^*(3)\leq 7$.  

\begin{theorem}\label{3case}
If $G$ is a $\{P_3\cup 2K_1,\overline{P_3\cup 2K_1}\}$-free graph with $\omega(G)=3$, then $\chi(G)\leq 7$.
\end{theorem}
\begin{proof}
Let $G$ be a $\{P_3\cup 2K_1,\overline{P_3\cup 2K_1}\}$-free graph with $\omega(G)=3$ and $V(G)=A\cup I_1\cup I_2\cup I_3\cup C_{1,2}\cup C_{1,3}\cup C_{2,3}$.
Let $\{1,2,\ldots,7\}$ be the set of colors.
We begin by observing the fact that $\langle C_{1,2}\cup C_{1,3}\cup I_1\rangle$ and $\langle C_{2,3}\cup I_2\rangle$ are  $(P_3\cup K_1)$-free (otherwise the $(P_3\cup K_1)\sqsubseteq \langle C_{1,2}\cup C_{1,3}\cup I_1\rangle ($ or $\langle C_{2,3}\cup I_2\rangle)$ together with $v_1 ($ or $v_2)$ forms a $P_3\cup 2K_1$, a contradiction).
One can also see that $\omega(\langle C_{1,2}\cup C_{1,3}\cup I_1\rangle)\leq 3$ and $\omega(\langle C_{2,3}\cup I_2\rangle)\leq 2$ and hence by Theorem \ref{chair}, $\chi(\langle C_{1,2}\cup C_{1,3}\cup I_1\rangle)\leq 4$ and $\chi(\langle C_{2,3}\cup I_2\rangle)\leq 3$.
To color the vertices in $V(G)$, we start by assigning the color $i$ to $v_i$, for $i=1,2,3$.
If $\chi(\langle C_{2,3}\cup I_2\rangle)\leq 2$, then we can show that $G$ is $7$-colorable by coloring the vertices in $I_3$, $(C_{2,3}\cup I_2)$ and $(C_{1,2}\cup C_{1,3}\cup I_1)$ with colors $3$, $\{2,4\}$ and $\{1,5,6,7\}$ respectively.
Hence let us consider the case when $\chi(\langle C_{2,3}\cup I_2\rangle)=3$.
Since $\langle C_{2,3}\cup I_2\rangle$ is $\{P_3\cup K_1,\overline{P_3\cup 2K_1}\}$-free with $\omega(\langle C_{2,3}\cup I_2\rangle)=2$ and $\chi(\langle C_{2,3}\cup I_2\rangle)=3$, by Theorem \ref{spgt}, there exists a $C_5\sqsubseteq \langle C_{2,3}\cup I_2\rangle$.
To show that $G$ is $7$-colorable it is enough to show that $\langle C_{2,3}\cup I_2\rangle\cong C_5$ and that there exists a vertex, say $a\in (C_{2,3}\cup I_2)$ such that $[a,U_3]=\emptyset$ (Here one can attain a $7$-coloring for $G$ by coloring the vertices in $(a\cup I_3)$, $(C_{2,3}\cup I_2)\backslash\{a\}$ and $(C_{1,2}\cup C_{1,3}\cup I_1)$ with colors $3$, $\{2,4\}$ and $\{1,5,6,7\}$ respectively).
Let us begin by showing that $\langle C_{2,3}\cup I_2\rangle\cong C_5$.
As mentioned earlier, there exists a $C_5\sqsubseteq \langle C_{2,3}\cup I_2\rangle$, let $V(C_5)=\{z_1,z_2,\ldots,z_5\}$.
Suppose there exists a vertex $b\in (C_{2,3}\cup I_2)\backslash V(C_5)$.
Since $\langle C_{2,3}\cup I_2 \rangle$ is $(P_3\cup K_1)$-free, it is not difficult to see that $|[b,V(C_5)]|\geq 2$.
If $b$ is adjacent to a pair of adjacent vertices in $C_5$ then $\omega(\langle v_1\cup b \cup V(C_5)\rangle)=4$, a contradiction.
Therefore, without loss of generality, let us assume that $bz_1,bz_3\in E(G)$, then $\langle\{b,z_1,z_2,z_4\}\rangle\cong P_3\cup K_1$, again a contradiction.
Hence $\langle C_{2,3}\cup I_2\rangle\cong C_5$.
Next, let us show that there exists a vertex in $(C_{2,3}\cup I_2)$ that has no neighbors in $U_3$.
On the contrary, let us assume that $V(C_5)\subseteq N(U_3)$.
Again, by using similar arguments one can see that each vertex in $U_3$ is adjacent to at most a non-adjacent pair of vertices in $C_5$.
Hence as a result we get that $|U_3|\geq 3$.
Let $y_1,y_2\in I_3$.
Note that $|V(C_5)\cap I_2|\leq 2$ (since $I_2$ is an independent set).
If $|V(C_5)\cap I_2|\leq 1$, then there exists a $P_4\in \langle C_{2,3}\cap V(C_5)\rangle$ such that $P_3\cup 2K_1\sqsubseteq\langle\{v_3,v_2,y_1,y_2,V(P_4)\}\rangle$, a contradiction.
Hence $\langle V(C_5)\cap I_2\rangle\cong 2K_1$ and as result $\langle V(C_5)\cap C_{2,3} \rangle\cong K_2\cup K_1$.
Let us assume that $(V(C_5)\cap I_2)=\{z_2,z_5\}$ and $(V(C_5)\cap C_{2,3})=\{z_1,z_3,z_4\}$.
One can see that $[v_3,\{z_1,z_3,z_4\}]=\emptyset$ 
$|[y_k,\{z_3,z_4\}]|\leq 1$, for $k=1,2$. 
 Thus $z_1y_1,z_1y_2\in E(G)$ (if say, $z_1y_1\notin E(G)$, then $P_3\cup 2K_1\sqsubseteq\langle\{y_1,v_2,v_3,z_1,z_3,z_4\}\rangle$, a contradiction.
We get a similar contradiction if $z_1y_2\notin E(G)$).
As a result, $[\{z_2,z_5\},\{y_1,y_2\}]=\emptyset$ and  $\langle\{z_2,v_3,z_5,y_1,y_2\}\rangle\cong P_3\cup 2K_1$, a contradiction.
Thus there exists a $z_{\ell}\in V(C_5)$ such that $[z_{\ell},U_3]=\emptyset$ and hence $G$ is 7-colorable.
\end{proof}
 
%

\bibliographystyle{ams}
\bibliography{ref}
\end{titlepage}
\listoftodos
\end{document}